\documentclass[12pt,reqno]{amsart}

\usepackage[colorlinks=true,linkcolor=blue,citecolor=blue,urlcolor=blue]{hyperref}
\usepackage{array}
\usepackage{amssymb,latexsym}
\usepackage{bm,bbding}
\usepackage{cleveref}

\newtheorem{theorem}{Theorem}[section]

\newtheorem{proposition}[theorem]{Proposition}
\newtheorem{corollary}[theorem]{Corollary}
\newtheorem{definition}[theorem]{Definition}

\numberwithin{equation}{section}
\theoremstyle{definition}

\newtheorem{example}[theorem]{Example}

\newcommand{\cp}{\mathbb{C}P}
\newcommand{\zz}{\mathbb{Z}}
\newcommand{\qq}{\mathbb{Q}}

\begin{document}

\title{The  Hirzebruch genera of complete intersections}{\thanks{This project was supported by the Natural Science Foundation of Tianjin City of China, No. 19JCYBJC30300.}}

\author{Jianbo Wang, Zhiwang Yu, Yuyu Wang}
\address{(Jianbo Wang) School of Mathematics, Tianjin University, Tianjin 300350, China}
\email{wjianbo@tju.edu.cn}
\address{(Zhiwang Yu) School of Mathematics, Tianjin University, Tianjin 300350, China}
\email{yzhwang@tju.edu.cn}
\address{(Yuyu Wang) College of Mathematical Science, Tianjin Normal University, Tianjin 300387, China}
\email{wdoubleyu@aliyun.com}

\begin{abstract}
Following Brooks's calculation of the $\hat{A}$-genus of complete intersections, a new and more computable formula about the $\hat{A}$-genus and $\alpha$-invariant will be described as polynomials of multi-degree and dimension. We also give an iterated formula of $\hat{A}$-genus and the necessary and sufficient conditions for the vanishing of $\hat{A}$-genus of complex even dimensional spin complete intersections. Finally, we obtain a general formula about the Hirzebruch genus of complete intersections, and calculate some classical Hirzebruch genera as examples. 
\end{abstract}

\date{\today}
\maketitle

\section{Introduction}

A genus of a multiplicative sequence is a ring homomorphism, from the ring of smooth compact manifolds (up to suitable cobordism) to another ring. A multiplicative sequence is completely determined by its characteristic power series $Q(x)$. Moreover, every power series $Q(x)$ determines a multiplicative sequence. Given a power series $Q(x)$, there is an associated Hirzebruch genus, which is also denoted as a $\varphi^{}_Q$-genus. For more details, please see \Cref{sec-genus}. The focus of this paper is mainly on the description of the Hirzebruch genus of complete intersections with multi-degree and dimension. 

A complex $n$-dimensional {\bf complete intersection}
$X_n(d_1,\dots,d_r)\subset \cp^{n+r}$ is a smooth complex $n$-dimensional manifold given by a transversal intersection of $r$ nonsingular hypersurfaces in
complex projective space. The unordered $r$-tuple $\underline{d}:=(d_1,\dots,d_r)$
is called the \emph{multi-degree}, which denotes the degrees of the $r$ nonsingular hypersurfaces, and the product $d:=d_1d_2\cdots d_r$ is
called the  \emph{total degree}. It is well-known that the diffeomorphism type of the real $2n$-dimensional manifold $X_n(d_1,\dots,d_r)$ depends only on the multi-degree and dimension. For the abbreviation of notation, set $X_n(\underline{d}):=X_n(d_1,\dots,d_r)$. By the Lefschetz hyperplane section theorem, the inclusion $X_n(\underline{d}) \subset \mathbb{C}P^{n+r}$ is $n$-connected. Let $x\in H^2(X_{n}(\underline{d});\zz)$ be the pullback of the first Chern class of the Hopf line bundle $H$ (the dual bundle of canonical line bundle) over $\cp^{n+r}$, the total Chern class and total Pontrjagin class of $X_{n}(\underline{d})$ (\cite{Libgober&Wood1982}) are given by
\begin{align}
c(X_{n}(\underline{d})) & =(1+x)^{n+r+1} \cdot \prod_{i=1}^{r}\left(1+d_{i} \cdot x\right)^{-1},\label{totalChern}\\
p(X_{n}(\underline{d})) & =\left(1+x^{2}\right)^{n+r+1} \cdot \prod_{i=1}^{r}\left(1+d_{i}^{2} \cdot x^{2}\right)^{-1}.\label{totalPont}
\end{align}
In particular, the first Chern class is 
\begin{equation}\label{c1p1}
c_{1}(X_{n}(\underline{d})) =\Big(n+r+1-\sum_{i=1}^r d_{i}\Big) \cdot x.
\end{equation}
In addition, one can show that the evaluation of $x^n$ on the fundamental class of $X_n(\underline{d})$ is the total degree $d=d_1\cdots d_r$:
\[
\int_{X_n(\underline{d})}x^n=d.
\]
 Note that, for a complete intersection $X_{n}(\underline{d})$,  the $i$-th Pontrjagin class $p_{i}$ must be an integral multiple of $x^{2i},$ where $x \in H^{2}(X_{n}(\underline{d}); \mathbb{Z}) \cong \mathbb{Z}$ is a generator ($n\ne 2$). This multiple is independent of the choice of the generator $x$ since $p_{i}\in H^{4i}(X_{n}(d); \mathbb{Z})$, so we can compare the Pontrjagin classes of complete intersections with the same dimension and different multi-degrees. Related to the diffeomorphism classification of complete intersections, there is the following conjecture, often called the Sullivan conjecture after Dennis Sullivan. 

\vskip 2mm
\noindent {\bf Sullivan conjecture}. When $n \ne 2$, two different complete intersections $X_n(\underline{d})$ and $X_n(\underline{d}^\prime)$ are diffeomorphic if and only if they have the same total degree, Pontrjagin classes and Euler characteristics.

\vskip 2mm
The partial known results on the Sullivan conjecture for complete intersections can be found in \cite{FK, FW, K, T}. The latest progress by Crowley and Nagy proves the Sullivan conjecture for $n = 4$ in \cite{CrowNa2020}.

Brooks \cite{Br1983} calculated the $\hat{A}$-genus of complex hypersurfaces $X_{2n}(d_1)$ and complete intersections $X_{2n}(\underline{d})$ with $\underline{d}=(d_1,\dots,d_r)$:
\begin{align}
\hat{A}(X_{2n}(d_1)) & =\frac{2}{(2n+1)!}\prod_{j=-n}^n\left(\frac{d_1}{2}-j\right), \label{Ahathypersurface} \\
\hat{A}(X_{2n}(\underline{d})) & =\frac{1}{2\pi\sqrt{-1}}\oint_{\varGamma(0)}(1+z)^{n-\frac{1}{2}-\sum\limits_{i=1}^r{\frac{d_i-1}{2}}}\cdot\frac{\prod_{i=1}^r{((1+z)^{d_i}-1)}}{z^{2n+r+1}}d z, \label{AhatCI}
\end{align}
where $\varGamma(0)$ denotes a small circle around $0$, and $\displaystyle\frac{1}{2\pi\sqrt{-1}}\oint_{\varGamma(0)}f(z)d z$ means the residue of $f(z)$ at $z=0$. Note that the $\hat{A}$-genus of complex odd dimensional complete intersection is zero.
\eqref{Ahathypersurface} can also be found in \cite[page 298]{Lawson1989}. Evidently,  \eqref{Ahathypersurface} is better relative to \eqref{AhatCI}. One starting point for writing this paper is to give an explicit description of $\hat{A}(X_{2n}(\underline{d}))$, which might be simpler and more computable than \eqref{AhatCI}. After all, \eqref{AhatCI} looks more cumbersome. The first main theorem in this paper is as follows:
\begin{theorem}\label{main-Ahatgenus}
The $\hat{A}$-genus of complete intersection $X_{2n}(\underline{d})=X_{2n}(d_1,\dots,d_r)$ is
\begin{equation*}
\hat{A}(X_{2n}(\underline{d}))=\sum_{j=0}^r(-1)^{r-j}\sum_{1\leqslant{k_1}<\dots<{k_j}\leqslant r}{\dbinom{\frac{1}{2}c_1-1 +d_{k_1}+\dots+d_{k_j}}{2n+r}},
\end{equation*}
where $c_1=2n+r+1-\sum\limits_{i=1}^r{d_i}$, i.e. $c_1\cdot x$ is the first Chern class $c_1(X_{2n}(\underline{d}))$ as in \eqref{c1p1},  $\dbinom{a}{k}:=\displaystyle\frac{a(a-1)(a-2)\cdots(a-k+1)}{k!}$, and $d_{k_1}+\dots+d_{k_j}$ vanishes  if $j=0$.
\end{theorem}
Combine the above theorem, we can give the $\alpha$-invariant of complete intersections $X_n(d_1,\dots,d_r)$ with $c_1=n+r+1-\sum\limits_{i=1}^rd_i$ even.
\begin{theorem}\label{main-alphainvariant}
Let $\underline{d}=(d_1,\dots,d_r)$. Then for a complete intersection $X_n(\underline{d})$ with $c_1=n+r+1-\sum\limits_{i=1}^rd_i$ even, the $\alpha$-invariant $\alpha(X_n(\underline{d}))$ is equal to
\begin{enumerate}
\item $\hat{A}(X_n(\underline{d}))$, if $n= 0\pmod 4$;
\item $\displaystyle\sum_{0\leqslant j\leqslant r-1\atop 1\leqslant k_1<\dots<k_j\leqslant r-1}\dbinom {\frac{1}{2}c_1-1+d_{k_1}+\dots+d_{k_j}}{n+r}\pmod 2$, if $n= 1\pmod 4$;
\item $\frac{1}{2}\hat{A}(X_n(\underline{d}))$, if $n= 2\pmod 4$;
\item $0$, if $n= 3\pmod 4$.
\end{enumerate}
\end{theorem}
Note that for $n>1$, the complete intersection $X_n(\underline{d})$ is spin if and only if $c_1$ is even. The complex one dimensional complete intersection $X_1(\underline{d})$ is a closed orientable surface with the second Stiefel-Whitney class $\omega_2=0$, and hence is spin. 

In \cite{Hi1978}, Hirzebruch defined the virtual index and virtual generalized Todd genus for a virtual submanifold $(v_{1}, \dots, v_{r})$ of $M$, where $v_{1}, \dots, v_{r}\in H^2(M;\zz)$, and $M$ is a compact oriented manifold or a compact almost complex manifold. Similarly, we define the virtual Hirzebruch genus of $(v_{1}, \dots, v_{r})$, which is equal to the Hirzebruch genus of a submanifold $V$ with codimension $2r$ of $M$. As an application, an iterated formula about $\hat{A}$-genus of complete intersections is given as follows:
\begin{theorem}\label{mainAhatIM}
For any positive integers $d_1,\dots,d_r$, assume that $d_{r-1}\geqslant d_r\geqslant 2$, then we have
\begin{equation*}
\hat{A}(X_{2n}(d_1,\dots,d_{r-1},d_r))=\sum_{k=0}^{d_r-1}\hat{A}(X_{2n}(d_1,\dots,d_{r-2},d_{r-1}+d_r-1-2k)).
\end{equation*}
\end{theorem} 
Brooks shows that, for a spin complete intersections $X_{2n}(d_1,\dots,d_r)$, its $\hat{A}$-genus vanishes if and only if $c_1> 0$ (\cite{Br1983}), where $c_1=2n+r+1-\sum\limits_{i=1}^rd_i$. We can use \Cref{mainAhatIM} to give a new proof for the vanishing of $\hat{A}(X_{2n}(d_1,\dots,d_r))$.

As another application of virtual Hirzebruch genus, we give a general description of any Hirzebruch genus of complete intersections:
\begin{theorem}\label{main-VHgenusCI}
For the given power series  $Q(x)=x/R(x)$, the Hirzebruch genus $\varphi^{}_Q$ on the complete intersection $X_n(\underline{d})$ with $\underline{d}=(d_1,\dots,d_r)$ satisfies that:
\begin{equation*}
\varphi^{}_Q(X_n(\underline{d}))=\frac{1}{2\pi\sqrt{-1}}\oint_{\varGamma(0)}\frac{\prod_{i=1}^r{R(d_iz)}}{{(R(z))}^{n+r+1}}d z.
\end{equation*}
\end{theorem}

When this paper is near completion, we notice Baraglia's paper \cite{Baraglia2020}.  Baraglia also give the similar description of the $\hat{A}$-genus and alpha invariant of complete intersections by the different methods. Ours results in \Cref{main-Ahatgenus} and \ref{main-alphainvariant} are equivalent to the results in \cite[Theorem 1.3 \& 1.2]{Baraglia2020}.

This paper is organized as follows. In \Cref{sec-genus}, we introduce some necessary concepts and properties on Hirzebruch genus. In \Cref{sec-A-genus}, we prove \Cref{main-Ahatgenus} and \Cref{main-alphainvariant}. In \Cref{virtualgenus}, the virtual Hirzebruch genus is defined. As an application of virtual genus,  we prove \Cref{mainAhatIM} and determine when the $\hat{A}$-genus of complex even dimensional spin complete intersections vanishes. In \Cref{sec-Hir-genus},  \Cref{main-VHgenusCI} is proved by using the virtual Hirzebruch genus of complete intersections, and we calculate some classical Hirzebruch genera of complete intersections as examples.

\section{Hirzebruch genus}\label{sec-genus}

For references about the Hirzebruch genus, \cite[Appendix E]{BuchPan15}, \cite{Hi1978} and \cite[\S 1.6-1.8]{HiBeJu1992}  are recommended.
Unless otherwise stated, all the manifolds mentioned in this paper are smooth, oriented, closed, and connected. A general real $n$-dimensional manifold is denoted by $M^n$. We omit the real dimension $n$ when there is no danger of ambiguity. If $M$ admits an almost complex structure, it must be even dimensional. For an almost complex manifold $M^{2n}$,  let $TM$ be the complex tangent $n$-vector bundle, and $(TM)_\mathbb{R}$ be the underlying  real  tangent $2n$-vector bundle of $M$. For a differentiable manifold $M^{n}$,  $TM$ is the tangent $n$-vector bundle of $M$.

Every homomorphism $\varphi: \Omega^U\to \Lambda$ from the complex bordism ring to a commutative ring $\Lambda$ with unit can be regarded as a multiplicative characteristic of manifolds which is an invariant of bordism classes. Such a homomorphism is called a {\bf {(complex) $\Lambda$-genus}}. The term “multiplicative genus” is also used, to emphasize that such a genus is a ring homomorphism.

Assume that $\Lambda$ does not have additive torsion, then every $\Lambda$-genus is fully determined by the corresponding homomorphism $\Omega^U\otimes\mathbb{Q}\rightarrow \Lambda\otimes\mathbb{Q}$, which is also denoted by $\varphi$. Every homomorphism $\varphi: \Omega^U\otimes\mathbb{Q}\rightarrow \Lambda\otimes\mathbb{Q}$ can be interpreted as a sequence of homogeneous polynomials $\{K_i(c_1,\dots,c_i),i\geqslant0\}$, $\deg K_i=2i$, with certain imposed conditions. The conditions may be described as follows: an identity
\[
1+c_1+c_2+\cdots=(1+c_1^\prime+c_2^\prime+\cdots)\cdot(1+c_1^{\prime\prime}+c_2^{\prime\prime}+\cdots)
\]
implies that
\begin{equation}\label{MS}
\sum_{n\geqslant0}K_n(c_1,\dots,c_n)=\sum_{i\geqslant0}K_i(c_1^\prime,\dots,c_i^\prime)\cdot\sum_{j\geqslant0}K_j(c_1^{\prime\prime},\dots,c_j^{\prime\prime}).
\end{equation}
A sequence of homogeneous polynomials $\mathcal{K}=\{K_i(c_1,\dots,c_i),i\geqslant0\}$ with $K_0=1$ is called {\bf {multiplicative Hirzebruch sequence}} if $\mathcal{K}$ satisfies the \eqref{MS}. Write the following abbreviated notation:
\[
K(c_1,c_2,\dots):=1+K_1(c_1)+K_2(c_1,c_2)+\cdots,
\]
then \eqref{MS} is
\[
K(c_1,c_2,\dots)=K(c_1^\prime,c_2^\prime,\dots)\cdot K(c_1^{\prime\prime},c_2^{\prime\prime},\dots).
\]
Moreover, a multiplicative sequence  $\mathcal{K}$ is one-to-one correspondence with a power series $Q(x)=1+q_1x+q_2x^2+\cdots\in \Lambda\otimes\mathbb{Q}[[x]]$, where the notation $\Lambda\otimes\qq[[x]]$ means the ring of power series in $x$ over the ring $\Lambda\otimes\qq$, $x=c_1$, and $q_i=K_i(1,0,\dots,0)$ (see \cite[Appendix E]{BuchPan15} or \cite[\S 1]{Hi1978}), i.e. $Q(x)=1+K_1(x,0,\dots,0)+K_2(x,0,\dots,0)+\cdots$.

By \eqref{MS}, an equality
\[
1+c_1+\cdots+c_n=(1+x_1)\cdots(1+x_n)
\]
implies the equality
\begin{align*}
Q(x_1)\cdots Q(x_n) & =1+K_1(c_1)+\cdots+K_n(c_1,\dots,c_n)+K_{n+1}(c_1,\dots,c_n,0)+\cdots\\
& = K(c_1,\dots,c_n).
\end{align*}
It follows that the $n$-th term $K_{n}(c_{1}, \ldots, c_{n})$ in the multiplicative Hirzebruch sequence $\mathcal{K}=\{K_i(c_1,\dots,c_i)\}$ corresponding to a genus $\varphi: \Omega^{U} \otimes \mathbb{Q} \rightarrow \Lambda\otimes \mathbb{Q}$ is the degree $2n$ part of the series $\prod_{i=1}^{n} Q(x_i) \in \Lambda\otimes \mathbb{Q}[[c_{1}, \ldots, c_{n}]]$. Follow the statement in \cite[Appendix E]{BuchPan15},  the series $\prod_{i=1}^{n}Q(x_i)$ can be regarded as a universal characteristic class of a complex $n$-vector bundles.
For example, for a complex $n$-vector bundle $\xi$ with the total Chern class $c(\xi)=1+c_1+\cdots+c_n=(1+x_1)\cdots (1+x_n)$, where $x_1,\dots,x_n$ are the formal roots of $c(\xi)$, we can define the $\bm{\varphi^{}_Q}$-{\bf class} of $\xi$ by
\begin{equation}
\varphi^{}_Q(\xi)=K(c_1,\dots,c_n)=\prod\limits_{i=1}^nQ(x_i).
\end{equation}
\begin{definition}\label{varphiQclassC}
Let $Q(x)$ be a power series with constant term being one.
For an  almost complex manifold $M^{2n}$, let $x_1,\dots,x_n$ be the formal roots of the total Chern class of $TM$, i.e., $c(TM)=(1+x_1)\cdots (1+x_n)$. The {\bf{Hirzebruch genus}} $\varphi^{}_Q$ corresponding to $Q(x)$  is defined by
\begin{equation*}
\varphi^{}_Q(M)=\int_M\varphi^{}_Q(TM)=\int_M\prod_{i=1}^nQ(x_i),
\end{equation*}
where $\int_M\varphi^{}_Q(TM)=\left\langle\varphi^{}_Q(TM),[M]\right\rangle$ is the evaluation of the degree $2n$ part of $\varphi^{}_Q(TM)$ on the fundamental class $[M]$ of $M$.
\end{definition}
A parallel theory of genera exists for oriented manifolds.
These genera are homomorphisms $\Omega^{SO}\rightarrow\Lambda$ from the oriented bordism ring to $\Lambda$.
Once again we assume that the ring $\Lambda$ does not have additive torsion.
Then every oriented $\Lambda$-genus $\varphi$ is determined by the corresponding homomorphism $\Omega^{SO}\otimes\mathbb{Q}\rightarrow \Lambda\otimes\mathbb{Q}$, which we also denote by $\varphi$.
Such genera are in one-to-one correspondence with even power series $Q(x)=1+q_2x^2+q_4x^4+\cdots\in\Lambda\otimes\mathbb{Q}[[x]]$.

Similarly as the complex case above, for a real $n$-vector bundle $\xi$ with the total Pontrjagin class $p(\xi)=1+p_1+\cdots+p^{}_{\left[\frac{n}{4}\right]}=(1+x_1^2)(1+x_2^2)\cdots (1+x_{\left[\frac{n}{4}\right]}^2)$, we can define the $\bm{\varphi^{}_Q}$-{\bf class} of $\xi$ by
\begin{equation}\label{phiQclassPontr}
\varphi^{}_Q(\xi)=K(p_1,\dots,p^{}_{\left[\frac{n}{4}\right]})=\prod\limits_{i=1}^{\left[\frac{n}{4}\right]}Q(x_i).
\end{equation}
\begin{definition}\label{varphiQclassP}
For a compact, oriented, differentiable manifold $M$ of dimension $n$,  the {\bf Hirzebruch genus} $\varphi^{}_Q$ corresponding to an even power series $Q(x)$ is defined by
\[
\varphi^{}_Q(M)=\int_M\varphi^{}_Q(TM)=\int_M\prod\limits_{i=1}^{\left[\frac{n}{4}\right]}Q(x_i),
\]
where the $x_1^2,\dots,x_{\left[\frac{n}{4}\right]}^2$ are the formal roots of the total Pontrjagin class of $TM$, i.e., $p(TM)=(1+x_1^2)\cdots (1+x_{\left[\frac{n}{4}\right]}^2)$.
\end{definition}
	
An oriented genus $\varphi: \Omega^{SO}\rightarrow\Lambda$ defines a complex genus given by the composition $\Omega^{U}\rightarrow\Omega^{SO}\stackrel{\varphi}{\longrightarrow}\Lambda$ with the homomorphism $\Omega^{U}\rightarrow\Omega^{SO}$ forgetting the stably complex structure.  Since $\Omega^{U}\rightarrow\Omega^{SO}$ is onto modulo torsion, and $\Lambda$ is assumed to
be torsion-free, one loses no information by passing from an oriented genus to a
complex one. On the other hand, a complex genus $\varphi:\Omega^{U}\rightarrow\Lambda$ factors through an
oriented genus $\Omega^{SO}\rightarrow\Lambda$ whenever its corresponding power series $Q(x)$ is even.

\section{The $\hat{A}$-genus  and $\alpha$-invariant of complete intersections}\label{sec-A-genus}

For a compact oriented differentiable manifold  $M^{4k}$, Hirzebruch (\cite[Page 197]{Hi1978}) defined the $\hat{A}$-sequence of $M^{4k}$ as a certain polynomials in the Pontrjagin classes of $M^{4k}$. More concretely, the power series
\[
Q(x)=\frac{\frac{1}{2}x}{\sinh \left(\frac{1}{2}x\right)}=1-\frac{x^2}{24}+\frac{7x^4}{5760}-\frac{31x^6}{967680}+\frac{127x^8}{154828800}+\cdots
\]
defines a multiplicative sequence $\left\{\hat{A}_{j}(p_{1}, \ldots, p_{j})\right\}$. For example, the first four terms
of the sequence are as follows (for abbreviation, $p_j:=p_j(M)=p_j(TM)$):
\begin{align*}
\hat{A}_{1} & =-\frac{p_{1}}{24}, \\
\hat{A}_{2} & =\frac{7p_{1}^{2}-4 p_{2}}{5760}, \\
\hat{A}_{3} & =\frac{-31 p_{1}^{3}+44 p_{1} p_{2}-16 p_{3}}{967680}, \\
\hat{A}_{4} & =\frac{381 p_{1}^{3}-904 p_{1}^{2} p_{2}+208 p_{2}^{2}+512 p_{1} p_{3}-192 p_{4}}{464486400}.
\end{align*}
The $\hat{A}$-genus $\hat{A}(M)$ is the $\varphi^{}_Q$-genus with the power series $Q(x)$ as above, and 
\[
\hat{A}(M)=\int_M \hat{A}(p_1,\dots,p_k),
\]
where $\hat{A}(p_1,\dots,p_k)=\varphi^{}_Q(TM)$ is the $\hat{A}$-class defined as in  \Cref{sec-genus}.

It is a result of Lichnerowicz that the $\hat{A}$-genus must vanish for a compact spin manifold which admits a Riemannian metric with positive scalar curvature (see \cite{Lichnerowicz1963}).  In \cite{Hitchin1974}, Hitchin generalized this result to $\alpha(M) = 0$ if $M$ has a metric of positive scalar curvature, where $\alpha: \Omega^{Spin}_*\to KO^{-*}(pt)$ is a homomorphism, and the $\alpha$-invariant $\alpha(M)\in KO^{-n}(pt)$ is a $KO$-characteristic number for a spin manifold $M$ of dimension $n$. This is a strict generalization of the result of Lichnerowicz since $\alpha(M)=\hat{A}(M)$  if $n= 0 \pmod 8$ and $\alpha(M)=\frac{1}{2}\hat{A}(M)$ if $n= 4 \pmod 8$.
Furthermore, Stolz showed that a spin manifold $M$ of dimension $n\geqslant 5$ admits a positive scalar curvature metric if and only if $\alpha(M)$ vanishes (see \cite{Stolz1990}), which proved the Gromov-Lawson conjecture.  By virtue of the $\alpha$-invariants of complex $4k+1$-dimensional complete intersection $X_{4k+1}(\underline{d})$ ($k\geqslant 1$) and Seiberg-Witten theory in complex dimension 2, Fang and Shao in \cite{Fang&Shao} gave a classification list of complete intersections admitting Riemannian metrics with positive scalar curvature.

The $\hat{A}$-genus is a special case of elliptic genera. There are many results concerning the $\hat{A}$-genus which are related with group actions (see \cite{Kawakubo1991}).
From the Atiyah-Singer index theorem, the $\hat{A}(M)$ can be interpreted as the index of the Dirac operator acting on spin-bundles over $M$.
A profound development of the classical result by Atiyah-Hirzebruch (see \cite{AtHi70}): If $M$ is connected spin and admits a nontrivial smooth $S^1$-action, then the $\hat{A}$-genus $\hat{A}(M)$ vanishes.

Based on the Brooks's calculation \eqref{AhatCI}, we can obtain a more computable formula about the $\hat{A}$-genus of complex even dimensional complete intersection.
\begin{theorem}\label{Ahatgenus}
The $\hat{A}$-genus of complete intersection $X_{2n}(\underline{d})=X_{2n}(d_1,\dots,d_r)$ is
\begin{equation*}
\hat{A}(X_{2n}(\underline{d}))=\sum_{j=0}^r(-1)^{r-j}\sum_{1\leqslant{k_1}<\dots<{k_j}\leqslant r}{\dbinom{\frac{1}{2}c_1-1 +d_{k_1}+\dots+d_{k_j}}{2n+r}},
\end{equation*}
where $c_1=2n+r+1-\sum\limits_{i=1}^r{d_i}$, i.e. $c_1\cdot x$ is the first Chern class $c_1(X_{2n}(\underline{d}))$ as in \eqref{c1p1},  $\dbinom{a}{k}:=\displaystyle\frac{a(a-1)(a-2)\cdots(a-k+1)}{k!}$, and  $d_{k_1}+\dots+d_{k_j}$ vanishes  if $j=0$.
\end{theorem}
\begin{proof}
Let $\varGamma(0)$ denote a small circle around $0$. By \eqref{AhatCI}
\begin{align*}
&~~~~ \hat{A}(X_{2n}(\underline{d}))=\frac{1}{2\pi\sqrt{-1}}\oint_{\varGamma(0)}(1+z)^{n-\frac{1}{2}-\sum\limits_{i=1}^r{\frac{d_i-1}{2}}}\cdot\frac{\prod\limits_{i=1}^r{((1+z)^{d_i}-1)}}{z^{2n+r+1}}d z \\
&= \frac{1}{2\pi\sqrt{-1}}\oint_{\varGamma(0)}\frac{1}{z^{2n+r+1}}\cdot\sum_{j=0}^r{{(-1)^{r-j}}\sum_{1\leqslant{k_1}<\dots<{k_j}\leqslant r}{{(1+z)^{n-\frac{1}{2}-\sum\limits_{i=1}^r{\frac{{d_i}-1}{2}}+d_{k_1}+\dots+d_{k_j}}}}}d z \\
&= \frac{1}{2\pi\sqrt{-1}}\oint_{\varGamma(0)}\frac{1}{z^{2n+r+1}}\cdot\sum_{j=0}^r{{(-1)^{r-j}}\sum_{1\leqslant{k_1}<\dots<{k_j}\leqslant r}{{(1+z)^{\frac{1}{2}(2n+r+1-\sum\limits_{i=1}^r{d_i})-1 +d_{k_1}+\dots+d_{k_j}}}}}d z \\
&= \frac{1}{2\pi\sqrt{-1}}\oint_{\varGamma(0)}\frac{1}{z^{2n+r+1}}\cdot\sum_{j=0}^r{{(-1)^{r-j}}\sum_{1\leqslant{k_1}<\dots<{k_j}\leqslant r}{{(1+z)^{\frac{1}{2}c_1-1+d_{k_1}+\dots+d_{k_j}}} }}d z.
\end{align*}
Thus, by the residue theorem, the $\hat{A}$-genus of $X_{2n}(\underline{d})$ is the coefficient of $z^{2n+r}$ in the double summation $\sum\limits_{j=0}^r{(-1)^{r-j}}\sum\limits_{1\leqslant{k_1}<\dots<{k_j}\leqslant r}{{(1+z)^{\frac{1}{2}c_1-1+d_{k_1}+\dots+d_{k_j}}}}$. We have
\begin{equation*}
\hat{A}(X_{2n}(\underline{d}))=\sum_{j=0}^r{{(-1)^{r-j}}\sum_{1\leqslant{k_1}<\dots<{k_j}\leqslant r}{\dbinom{\frac{1}{2}c_1-1 +d_{k_1}+\dots+d_{k_j}}{2n+r}}}.
\end{equation*}
Thus  \Cref{Ahatgenus} is done.
\end{proof}
We can give another proof of \Cref{Ahatgenus} by calculating the $\hat{A}$-genus directly. 
\begin{proof}[2nd Proof of \Cref{Ahatgenus}]
Firstly, for any almost complex manifold $M$, the $\hat{A}$-class of the underlying real tangent bundle $(TM)_\mathbb{R}$ is
\begin{equation}\label{ahatclass}
\hat{A}((TM)_\mathbb{R})=\sum_{k=0}^\infty{\hat{A}_k(p_1,\dots,p_k)}.
\end{equation}
By \eqref{totalPont} and \eqref{phiQclassPontr}, the $\hat{A}$-class of $(T(X_{n}(\underline{d})))_\mathbb{R}$ is
\begin{align}
\hat{A}((T(X_{n}(\underline{d})))_\mathbb{R})=& \prod_{i=1}^r\frac{\sinh\left(\frac{d_i}{2}x\right)}{\frac{d_i}{2}x}\cdot\left(\frac{\frac{1}{2}x}{\sinh\left(\frac{1}{2}x\right)}\right)^{n+r+1} \nonumber \\
=& \prod_{i=1}^r\frac{\exp\left(\frac{d_i}{2}x\right)-\exp\left(-\frac{d_i}{2}x\right)}{d_ix}\cdot\left(\frac{x}{\exp\left(\frac{1}{2}x\right)-\exp\left(-\frac{1}{2}x\right)}\right)^{n+r+1} \nonumber \\
=& \frac{1}{d}x^{n+1}\cdot\frac{\prod_{i=1}^r{\left(\exp\left(\frac{d_i}{2}x\right)-\exp\left(-\frac{d_i}{2}x\right)\right)}}{{\left(\exp\left(\frac{1}{2}x\right)-\exp\left(-\frac{1}{2}x\right)\right)}^{n+r+1}} \nonumber \\
=& \frac{1}{d}x^{n+1}\cdot\exp\left(\frac{1}{2}(n+r+1)x-\frac{1}{2}\sum_{i=1}^rd_ix\right)\cdot\frac{\prod_{i=1}^r{\left(\exp\left(d_ix\right)-1\right)}}{{\left(\exp\left(x\right)-1\right)}^{n+r+1}} \nonumber \\
=& \frac{1}{d}x^{n+1}\cdot\exp\left(\frac{1}{2}c_1x\right)\cdot\frac{\prod_{i=1}^r{\left(\exp\left(d_ix\right)-1\right)}}{{\left(\exp\left(x\right)-1\right)}^{n+r+1}},\label{AhatclassCI}
\end{align}
where $d=d_1\cdots d_r$ is the total degree of $X_n(\underline{d})$ and $c_1=n+r+1-\sum\limits_{i=1}^rd_i$. 

By \eqref{AhatclassCI}, we have 
\begin{equation*}
\hat{A}((T(X_{2n}(\underline{d})))_\mathbb{R})=\frac{1}{d}x^{2n+1}\cdot\exp\left(\frac{1}{2}c_1x\right)\cdot\frac{\prod_{i=1}^r{\left(\exp\left(d_ix\right)-1\right)}}{{\left(\exp\left(x\right)-1\right)}^{2n+r+1}}. 
\end{equation*}
Then 
\begin{align*}
\hat{A}(X_{2n}(\underline{d}))=& \int_{X_{2n}(\underline{d})} \hat{A}((T(X_{2n}(\underline{d})))_\mathbb{R}) \\
=& \int_{X_{2n}(\underline{d})} \frac{1}{d}x^{2n+1}\cdot\frac{\exp\left(\frac{1}{2}c_1x\right)}{(\exp(x)-1)^{2n+r+1}}\prod_{i=1}^r(\exp(d_ix)-1). 
\end{align*}

Since what we are interested in is some coefficient in a polynomial, we use residue theorem to calculate this integral. Notice that $\int_{X_{2n}(\underline{d})}x^{2n}=d$ is the total degree of the complete intersection $X_{2n}(\underline{d})$, then we have 
\begin{align*}
\hat{A}(X_{2n}(\underline{d}))&= \frac{1}{2\pi\sqrt{-1}}\oint_{\varGamma(0)}\frac{\exp\left(\frac{1}{2}c_1z\right)}{(\exp(z)-1)^{2n+r+1}}\prod_{i=1}^r(\exp(d_iz)-1)d z \\
&= \frac{1}{2\pi\sqrt{-1}}\oint_{\varGamma(0)}\frac{\exp\left((\frac{1}{2}c_1-1)z\right)}{(\exp(z)-1)^{2n+r+1}}\prod_{i=1}^r(\exp(d_iz)-1)d\exp(z).
\end{align*}
By means of variable substitution: $\omega=\exp(z)-1$,
\begin{align*}
\hat{A}(X_{2n}(\underline{d}))&= \frac{1}{2\pi\sqrt{-1}}\oint_{\varGamma(0)}\frac{(1+\omega)^{\frac{1}{2}c_1-1}}{\omega^{2n+r+1}}\prod_{i=1}^r((1+\omega)^{d_i}-1)d \omega\\
&= \sum_{j=0}^r{{(-1)^{r-j}}\sum_{1\leqslant{k_1}<\dots<{k_j}\leqslant r}\dbinom{\frac{1}{2}c_1-1 +d_{k_1}+\dots+d_{k_j}}{2n+r}}.
\end{align*}
Again, \Cref{Ahatgenus} follows.
\end{proof}
The alpha invariant $\alpha(M)\in KO^{-n}(pt)$ is a $KO$-characteristic number for a spin $n$ dimensional manifold $M$, and $KO^{-n}(pt)$ is known by the Bott periodicity theorem:
\[
\begin{array}{c|c|c|c|c|c|c|c|c}
\hline n\pmod 8 & {0} & {1} & {2} & {3} & {4} & {5} & {6} & {7} \\
\hline K O^{-n}(pt) & {\mathbb{Z}} & {\mathbb{Z} / 2} & {\mathbb{Z} / 2} & {0} & {\mathbb{Z}} & {0} & {0} & {0} \\
\hline
\end{array}
\]
For a complex $n$-dimensional complete intersection $X_n(\underline{d})$, $\alpha(X_n(\underline{d}))\in KO^{-2n}(pt)$. Then
$\alpha(X_n(\underline{d}))$ vanishes if $n= 3\pmod 4$.
Combine \Cref{Ahatgenus} and the results in \cite{Fang&Shao}, we can obtain the $\alpha$-invariant of  complete intersections $X_n(d_1,\dots,d_r)$  with $c_1=n+r+1-\sum\limits_{i=1}^rd_i$ even.

\begin{theorem}\label{AIformula}
Let $\underline{d}=(d_1,\dots,d_r)$. Then for a complete intersection $X_n(\underline{d})$ with $c_1=n+r+1-\sum\limits_{i=1}^rd_i$ even, the $\alpha$-invariant $\alpha(X_n(\underline{d}))$ is equal to
\begin{enumerate}\itemsep=3mm
\item $\hat{A}(X_n(\underline{d}))$, if $n= 0\pmod 4$, and is equal to $\dfrac{1}{2}\hat{A}(X_n(\underline{d}))$, if $n= 2\pmod 4$.
\item $\displaystyle{\sum_{0\leqslant j\leqslant r-1 \atop 1\leqslant k_1<\dots<k_j\leqslant r-1}\dbinom {\frac{1}{2}c_1-1+d_{k_1}+\dots+d_{k_j}}{n+r}}\pmod 2$, if $n= 1\pmod 4$, where \\
$d_{k_1}+\dots+d_{k_j}$ vanishes  if $j=0$.

\item $0$, if $n= 3\pmod 4$.
\end{enumerate}
\end{theorem}
\begin{proof}
(1) follows directly from \cite[Proposition 4.2]{Hitchin1974}, and $\hat{A}(X_n(\underline{d}))$ is as in \Cref{Ahatgenus}.

(3) is trivial.

For (2), when $n= 1\pmod 4$, by \cite[(8)]{Fang&Shao}, the $\alpha$-invariant of $X_n(\underline{d})$ is as follows:
\begin{equation}\label{Fangshao}
\alpha(X_{n}(\underline{d}))=\sum\dbinom{\frac{1}{2}\left(n+r-1\pm d_{1}\pm\cdots\pm d_{r-1}+d_{r}\right)}{n+r} \pmod 2,
\end{equation}
where the summation sums over all the possibilities $\pm d_1\pm\cdots\pm d_{r-1}+d_r$.
By \cite[Remark 2]{Fang&Shao}, the choice of $d_r$ in \eqref{Fangshao} is not important to the result, since 
\[
\dbinom{\frac{n-1}{2}+k}{n}=\dbinom{\frac{n-1}{2}-k}{n}\pmod 2. 
\]
Therefore we can make \eqref{Fangshao} stepped forward:
\begin{align*}
\alpha(X_{n}(\underline{d}))= & \sum\dbinom{\frac{1}{2}\left(n+r-1\pm d_{1}\pm\cdots\pm d_{r-1}+d_{r}\right)}{n+r} \pmod 2 \\
= & \sum\dbinom{\frac{1}{2}\left(n+r+1\pm d_{1}\pm\cdots\pm d_{r-1}-d_{r}\right)-1}{n+r} \pmod 2\\
=& \sum\dbinom{\frac{1}{2}(n+r+1-\sum\limits_{i=1}^r{d_i}+d_1\pm d_{1}+d_2\pm d_2+\dots+d_{r-1}\pm d_{r-1})-1}{n+r} \\
=& \sum\dbinom{\frac{1}{2}c_1+\frac{1}{2}\left(d_1\pm d_{1}+d_2\pm d_2+\dots+d_{r-1}\pm d_{r-1}\right)-1}{n+r} \\
=& \sum_{\varepsilon_1,\dots,\varepsilon_{r-1}\in\{0,1\}}\dbinom{\frac{1}{2}c_1-1+\varepsilon_1d_1+\varepsilon_2d_2+\dots+\varepsilon_{r-1}d_{r-1}}{n+r} \\
=& \sum_{j=0}^{r-1}\sum_{j=\varepsilon_1+\dots+\varepsilon_{r-1}\atop\varepsilon_1,\dots,\varepsilon_{r-1}\in\{0,1\}}\dbinom{\frac{1}{2}c_1-1+\varepsilon_1d_1+\varepsilon_2d_2+\dots+\varepsilon_{r-1}d_{r-1}}{n+r} \\
=& \sum_{j=0}^{r-1}\sum_{1\leqslant k_1<\dots<k_j\leqslant r-1}\dbinom {\frac{1}{2}c_1-1+d_{k_1}+\dots+d_{k_j}}{n+r}.
\end{align*}
Thus, 
\begin{equation*}
\alpha(X_{n}(\underline{d}))=\sum_{0\leqslant j\leqslant r-1 \atop 1\leqslant k_1<\dots<k_j\leqslant r-1}\dbinom {\frac{1}{2}c_1-1+d_{k_1}+\dots+d_{k_j}}{n+r} \pmod 2.\qedhere
\end{equation*}
\end{proof}
In fact, the key formula \eqref{Fangshao} we used in the above proof is deduced from the following formula (\cite[(5)]{Fang&Shao}):
\begin{align*}
&~~~~\alpha(X_{n}(d_1,\dots,d_{r-1},d_r))
\\
&= \int_{X_{n+1}(d_1,\dots,d_{r-1})}\hat{A}((T(X_{n+1}(d_1,\dots,d_{r-1})))_\mathbb{R})\cdot\exp\left(\frac{d_r}{2}x\right) \pmod 2.
\end{align*}
Begin with this formula, we can directly prove \Cref{AIformula} (2).
\begin{proof}[2nd Proof of \Cref{AIformula} (2)]
\begin{align*}
& \alpha(X_{n}(d_1,\dots,d_{r-1},d_r))\\
=& \int_{X_{n+1}(d_1,\dots,d_{r-1})}\hat{A}((T(X_{n+1}(d_1,\dots,d_{r-1})))_\mathbb{R})\cdot\exp\left(\frac{d_r}{2}x\right) \pmod 2 \\ 
=& \int_{X_{n+1}(d_1,\dots,d_{r-1})}\hat{A}((T(X_{n+1}(d_1,\dots,d_{r-1})))_\mathbb{R})\cdot\exp\left(-\frac{d_r}{2}x\right) \pmod 2.
\end{align*}
By \eqref{AhatclassCI}, we have 
\begin{align*}
&\alpha(X_{n}(d_1,\dots,d_{r-1},d_r))\\
= & \int_{X_{n+1}(d_1,\dots,d_{r-1})}\frac{x^{n+2}}{\prod_{i=1}^{r-1}d_i}\cdot\frac{\prod_{i=1}^{r-1}\left(\exp{\left(\frac{d_i}{2}x\right)}-\exp{\left(-\frac{d_i}{2}x\right)}\right)}{\left(\exp{\left(\frac{1}{2}x\right)}-\exp{\left(-\frac{1}{2}x\right)}\right)^{n+r+1}}\cdot \exp{\left(-\frac{d_r}{2}x\right)}. 
\end{align*}
By residue theorem, 
\begin{align*}
&~~~~ \alpha(X_{n}(d_1,\dots,d_{r-1},d_r))\\
&= \frac{1}{2\pi\sqrt{-1}}\oint_{\varGamma(0)}\frac{\prod_{i=1}^{r-1}\left(\exp{\left(\frac{d_i}{2}z\right)}-\exp{\left(-\frac{d_i}{2}z\right)}\right)}{\left(\exp{\left(\frac{1}{2}z\right)}-\exp{\left(-\frac{1}{2}z\right)}\right)^{n+r+1}}\cdot \exp{\left(-\frac{d_r}{2}z\right)}d z \\
&= \frac{1}{2\pi\sqrt{-1}}\oint_{\varGamma(0)}\exp{\left(\frac{1}{2}(n+r+1)z\right)}\cdot\prod_{i=1}^{r}\exp{\left(-\frac{d_i}{2}z\right)}\cdot\frac{\prod_{i=1}^{r-1}(\exp{(d_iz)}-1)}{(\exp{(z)}-1)^{n+r+1}}d z \\
&= \frac{1}{2\pi\sqrt{-1}}\oint_{\varGamma(0)}\exp{\left(\frac{1}{2}(n+r+1-\sum_{i=1}^rd_i)z\right)}\cdot \frac{\prod_{i=1}^{r-1}(\exp{(d_iz)}-1)}{(\exp{(z)}-1)^{n+r+1}}d z \\
&= \frac{1}{2\pi\sqrt{-1}}\oint_{\varGamma(0)}\exp{\left(\frac{1}{2}c_1z\right)}\cdot \frac{\prod_{i=1}^{r-1}(\exp{(d_iz)}-1)}{(\exp{(z)}-1)^{n+r+1}}\cdot\exp(-z)d(\exp(z)-1).
\end{align*}
By means of variable substitution: $\omega=\exp(z)-1$,
\begin{align*}
&~~~~ \alpha(X_{n}(d_1,\dots,d_{r-1},d_r))\\
&= \frac{1}{2\pi\sqrt{-1}}\oint_{\varGamma(0)}(1+\omega)^{\frac{1}{2}c_1}\cdot\frac{\prod_{i=1}^{r-1}((1+\omega)^{d_i}-1)}{\omega^{n+r+1}}\cdot(1+\omega)^{-1}d \omega \\
&= \frac{1}{2\pi\sqrt{-1}}\oint_{\varGamma(0)}(1+\omega)^{\frac{1}{2}c_1-1}\cdot\frac{\prod_{i=1}^{r-1}((1+\omega)^{d_i}-1)}{\omega^{n+r+1}}d \omega \\
&= \sum_{0\leqslant j\leqslant r-1 \atop 1\leqslant k_1<\dots<k_j\leqslant r-1}\dbinom {\frac{1}{2}c_1-1+d_{k_1}+\dots+d_{k_j}}{n+r} \pmod 2. \qedhere
\end{align*}
\end{proof}

\section{The virtual Hirzebruch genus of complete intersections}\label{virtualgenus}
 
 Inspired by the virtual index and virtual generalized Todd genus in \cite[\S 9, \S 11]{Hi1978}, we define the virtual Hirzebruch genus. First of all, let's introduce the concept of virtual submanifold. The following part about the virtual submanifold is in  \cite[\S 3.1]{HiBeJu1992}
 
 According to Thom \cite{Thom1954}, every 2-dimensional integral cohomology class $v\in H^2(M^n;\mathbb{Z})$ of a compact oriented differentiable manifold $M^n$ can be represented by a submanifold $V^{n-2}$, i.e., the cohomology class $v$ is the Poincar\'e duality of  the fundamental class of the oriented submanifold $V^{n-2}$. Therefore $v\in H^2(M;\zz)$ is called a {\bf virtual submanifold}. Assume that $v_1, v_2, \dots, v_r\in H^2(M;\zz)$, and $v_1$ is a virtual submanifold represented by a submanifold $V^{n-2}$ of the manifold $M$, that the restriction of $v_2$ to $V^{n-2}$ represents a submanifold $V^{n-4}$ of $V^{n-2}$, $\dots$, and finally that the restriction of $v_r$ to $V^{n-2(r-1)}$ represents a submanifold $V^{n-2r}$ of $V^{n-2(r-1)}$. The manifold $V^{n-2r}$ is then a submanifold of $M$ of codimension $2r$, which represents the virtual submanifold $(v_{1}, \ldots, v_{r})$. An alternative construction is the following: If $v_{1}, \dots, v_{r}$ can be represented by codimension $2$ submanifolds $V_{1}^{\prime}, \ldots, V_{r}^{\prime}$ of $M$, which are transversal to one another, then their transversal intersection represents the virtual submanifold $(v_{1}, \dots, v_{r})$.

Let's set  up the following convention:
\begin{enumerate}
\item If $M$ is a compact almost complex manifold,  $Q(x)=x/R(x)$ is taken as any power series with constant term being one.
\item If $M$ is a compact  oriented differentiable manifold,  $Q(x)=x/R(x)$ is taken as an even power series with constant term being one.
\end{enumerate}

\begin{definition}\label{VK}
For the integral cohomology classes $v_{1}, v_{2}, \ldots, v_{r}$ of $H^{2}(M;\mathbb{Z})$, the {\bf virtual Hirzebruch genus} $\varphi^{}_Q$ of the virtual submanifold $(v_{1}, v_{2}, \ldots, v_{r})$ is defined as follows:
\begin{align*}
\varphi^{}_Q(v_1,\dots,v_r)_M& =\int_{M}{\prod_{j=1}^rR(v_j)\cdot\varphi^{}_Q(TM)}, 
\end{align*}
where $\varphi^{}_Q(TM)$ is the $\varphi^{}_Q$-class defined as in \Cref{sec-genus}.
\end{definition}
In fact, the virtual Hirzebruch genus of the virtual submanifold $(v_{1}, v_{2}, \ldots, v_{r})$ is related to the Hirzebruch genus of the submanifold $V^{n-2r}$ of $M^n$. Similar to the argument in \cite[\S 9, \S 11]{Hi1978}, we have the following result.
\begin{proposition}\label{VandRK}
As the notations in \Cref{VK}. Assume that the virtual submanifold $(v_{1}, \dots, v_{r})$ is represented by a submanifold $V^{n-2r}$ of the manifold $M^n$, then
\begin{equation*}
\varphi^{}_Q(v_1,\dots,v_r)_M=\varphi^{}_Q(V^{n-2r}).
\end{equation*}
\end{proposition}
\begin{proof}
We only give the proof when $M$ is a compact oriented differentiable $n$-manifold. This proof is adapted to the case of  almost complex manifold (see \cite[\S 11]{Hi1978}). 

 Let $v\in H^2(M^n;\mathbb{Z})$ be a virtual manifold represented by a submanifold $V^{n-2}$, and $j: V^{n-2}\rightarrow M^n$ be the embedding of the oriented submanifold $V^{n-2}$ of $M^n$. Let $TV^{n-2}$ and $TM^n$ be tangle bundles of $V^{n-2}$, $M^n$ respectively and $\mathcal{N}$ is the normal bundle of $V^{n-2}$ in $M^n$, then
\[
j^*(TM^n)=TV^{n-2}\oplus\mathcal{N}.
\]
Let $p(V^{n-2})$ and $p(M^n)$ be the total Pontrjagin classes of $V^{n-2}$ and $M^n$ respectively. Then
\[
j^*p(M^n)=p(V^{n-2})p(\mathcal{N}) \text{~modulo~torsion}.
\]
Since $p(\mathcal{N})=j^*(1+v^2)$, then
\[
p(V^{n-2})=j^*((1+v^2)^{-1}p(M^n)).
\]
Note that $(1+v^2)^{-1}$ is unique in the cohomology class of $M^n$. Thus, for every Hirzebruch genus $\varphi^{}_Q$ associated to an even power series $Q(x)$, we have
\[
\varphi^{}_Q(TV^{n-2})=j^*\left(\frac{R(v)}{v}\varphi^{}_Q(TM^n)\right).
\]
For a cohomology class $u\in H^{n-2}(M^n;\mathbb{Z})$, by Poincar\'e duality, we have
\[
\int_{V^{n-2}}j^*(u)=\int_{M^n}vu.
\]
Therefore,
\begin{align*}
\varphi^{}_Q(V^{n-2}) & =\int_{V^{n-2}}\varphi^{}_Q(TV^{n-2})= \int_{V^{n-2}}j^*\left(\frac{R(v)}{v}\varphi^{}_Q(TM^n)\right)\\
& =\int_{M^n}R(v)\varphi^{}_Q(TM^n)=\varphi^{}_Q(v)_M.
\end{align*}
By finite induction, it implies that $\varphi^{}_Q(v_1,\dots,v_r)_M=\varphi^{}_Q(V^{n-2r})$.
\end{proof}

Now we consider the complete intersection $X_{n}(d_1,\dots,d_r)$ embedded in $\cp^{n+r}$. Let $x\in H^2(\cp^{n+r};\zz)$ be the first Chern class of the Hopf line bundle $H$ over $\cp^{n+r}$, then the virtual submanifold $(d_1x,\dots, d_rx)$ is represented by the complete intersection $X_{n}(d_1,\dots, d_r)$.
So by \Cref{VandRK}, we have 
\begin{corollary}\label{VRKCI}
The $\varphi^{}_Q$-genus of $X_n(d_1,\dots,d_r)$ is the virtual $\varphi^{}_Q$-genus of the virtual submanifold $(d_1x,\dots, d_rx)$, i.e., 
\begin{equation*}
\varphi^{}_Q(d_1x,\dots,d_rx)_{\mathbb{C}P^{n+r}}=\varphi^{}_Q(X_{n}(d_1,\dots,d_r)). 
\end{equation*}
\end{corollary}

As a special case of \Cref{VK}, we can also define the virtual $\hat{A}$-genus.
\begin{definition}\label{VAhat}
Let $M$ be a compact oriented differentiable manifold. For $v_1,\dots,v_r\in H^2(M;\mathbb{Z})$, the {\bf virtual  $\hat{A}$-genus} of $(v_1,\dots, v_r)$ is defined as follows:
\begin{equation*}
\hat{A}(v_1,\dots,v_r)_M=\int_{M}{\prod_{j=1}^rR(v_j)\cdot\hat{A}(TM)},
\end{equation*}
where $Q(x)=\dfrac{x}{R(x)}=\dfrac{\frac{1}{2}x}{\sinh\left(\frac{1}{2}x\right)}$.
\end{definition}
\begin{proposition}\label{prop-VAhatadd}
Let $M$ be a compact oriented differentiable manifold. For any $v_1,\dots,v_r,w\in H^2(M;\zz)$, the virtual $\hat{A}$-genus satisfies the following equality:
\begin{align*}
&~~~~ \hat{A}(v_1,\dots,v_{r-2},v_{r-1}+w,v_r+w)_M \\
&= \hat{A}(v_1,\dots,v_{r-2},v_{r-1},v_r)_M+\hat{A}(v_1,\dots,v_{r-2},v_{r-1}+v_r+w,w)_M.
\end{align*}
\end{proposition}
\begin{proof}
The $\hat{A}$-genus is corresponding to the power series $Q(x)=\dfrac{x}{R(x)}=\dfrac{\frac{1}{2}x}{\sinh\left(\frac{1}{2}x\right)}$, and
$R(x)=2\sinh\left(\frac{1}{2}x\right)$ satisfies the following identity
\begin{equation*}
R(u+w)R(v+w)=R(u)R(v)+R(u+v+w)R(w).
\end{equation*}
By \Cref{VAhat}, for any $w\in H^2(M;\zz)$, it induces the following equality
\begin{align*}
&~~~~ \hat{A}(v_1,\dots,v_{r-2},v_{r-1}+w,v_r+w)_M \\
&= \int_{M}{\prod_{j=1}^{r-2}R(v_j)\cdot R(v_{r-1}+w)\cdot R(v_r+w)\hat{A}(TM)}\\
&= \int_{M}{\prod_{j=1}^{r}R(v_j)\cdot \hat{A}(TM)}+\int_{M}{\prod_{j=1}^{r-2}R(v_j)\cdot R(v_{r-1}+v_r+w)\cdot R(w)\hat{A}(TM)}\\
&= \hat{A}(v_1,\dots,v_{r-2},v_{r-1},v_r)_M+\hat{A}(v_1,\dots,v_{r-2},v_{r-1}+v_r+w,w)_M.\qedhere
\end{align*}
\end{proof}
By \Cref{VRKCI} and \Cref{prop-VAhatadd},  we have
\begin{corollary}\label{cor-AhatCI}
For any positive integers $d_1,\dots,d_r, e$, we have the following equality on the $\hat{A}$-genus of complete intersections
\begin{align*}
&~~~~\hat{A}(X_{2n}(d_1,\dots,d_{r-2},d_{r-1}+e,d_r+e))\\
&= \hat{A}(X_{2n}(d_1,\dots,d_{r-1},d_r))+\hat{A}(X_{2n}(d_1,\dots,d_{r-2},d_{r-1}+d_r+e,e)).
\end{align*}
\end{corollary}		
\begin{theorem}
For any positive integers $d_1,\dots,d_r$, assume that $d_{r-1}\geqslant d_r\geqslant 2$, then we have
\begin{equation}\label{AhatIM}
\hat{A}(X_{2n}(d_1,\dots,d_{r-1},d_r))=\sum_{k=0}^{d_r-1}\hat{A}(X_{2n}(d_1,\dots,d_{r-2},d_{r-1}+d_r-1-2k)).
\end{equation}
Note that, if $d_r=1$, 
$\hat{A}(X_{2n}(d_1,\dots,d_{r-1},1))=\hat{A}(X_{2n}(d_1,\dots,d_{r-1}))$.
\end{theorem}
\begin{proof}
Applying \Cref{cor-AhatCI}, then 
\begin{align*}
&~~~~ \hat{A}(X_{2n}(d_1,\dots,d_{r-2},d_{r-1},d_r))\\
&= \hat{A}(X_{2n}(d_1,\dots,d_{r-2},d_{r-1}-1,d_r-1))+\hat{A}(X_{2n}(d_1,\dots,d_{r-2},d_{r-1}+d_r-1,1))\\
&= \hat{A}(X_{2n}(d_1,\dots,d_{r-2},d_{r-1}-1,d_r-1))+\hat{A}(X_{2n}(d_1,\dots,d_{r-2},d_{r-1}+d_r-1)).
\end{align*}
Let's iterate the above process, the proof is finished.
\end{proof}
		
Consider complex even dimensional spin complete intersection $X_{2n}(\underline{d})$ with multi-degree $\underline{d}=(d_1,\dots,d_r)$. Note that $X_{2n}(\underline{d})$ is spin if and only if $c_1=2n+r+1-\sum\limits_{i=1}^r{d_i}= 0 \pmod 2$.
In \cite[Theorem 2]{Br1983}, Brooks can determine when $\hat{A}(X_{2n}(\underline{d}))=0$ for the spin complete intersections $X_{2n}(\underline{d})$. Now, we give a new proof for the necessary and sufficient conditions related to the vanishing of $\hat{A}(X_{2n}(\underline{d}))$.

\begin{theorem}\label{NSCAhatV}
For each complex even dimensional spin complete intersection $X_{2n}(\underline{d})$ with $\underline{d}=(d_1,\dots,d_r)$, the $\hat{A}$-genus of $X_{2n}(\underline{d})$ has the following properties:
\begin{enumerate}
\item $\hat{A}(X_{2n}(\underline{d}))=0$, if $c_1>0$;
\item $\hat{A}(X_{2n}(\underline{d}))>0$, if $c_1\leqslant0$,
\end{enumerate}
where $c_1=2n+r+1-\sum\limits_{i=1}^rd_i$.
\end{theorem}
\begin{proof}
Since the diffeomorphism type of $X_{2n}(\underline{d})$ is independent of the order of the degrees $d_1,\dots,d_{r-1},d_r$, without losing generality, we assume that $d_{r-1}\geqslant d_r$. 

We prove the theorem by induction on $r$.

For $r=1$, $\hat{A}(X_{2n}(d_1))=\frac{2}{(2n+1)!}\prod\limits_{j=-n}^n\left(\frac{d_1}{2}-j\right)$, $c_1=2n+2-d_1$, and this case is trivial.

For $r>1$, if $c_1>0$, then
\begin{align*}
& 2n+r+1-\sum_{i=1}^r{d_i}=2n+(r-1)+1-\left(\sum_{i=1}^{r-2}{d_i}+(d_{r-1}+d_r-1)\right)>0,\\
& 2n+(r-1)+1-\left(\sum_{i=1}^{r-2}{d_i}+(d_{r-1}+d_r-1-2k)\right)>0, 0\leqslant k\leqslant d_r-1.
\end{align*}
Thus, by induction hypothesis,
$$
\hat{A}(X_{2n}(d_1,\dots,d_{r-2},d_{r-1}+d_r-1-2k))=0, ~ 0\leqslant k\leqslant d_r-1,
$$
i.e. the right side of \eqref{AhatIM} is zero, then $\hat{A}(X_{2n}(d_1,\dots,d_r))=0$;

Conversely, if $c_1\leqslant0$, then
\begin{equation*}
2n+r+1-\sum_{i=1}^r{d_i}=2n+(r-1)+1-\left(\sum_{i=1}^{r-2}{d_i}+(d_{r-1}+d_r-1)\right)\leqslant0.
\end{equation*}
Thus, by induction hypothesis,
$$\hat{A}(X_{2n}(d_1,\dots,d_{r-2},d_{r-1}+d_r-1))>0,$$
and all other summation terms in \eqref{AhatIM} are nonnegative, so
$\hat{A}(X_{2n}(d_1,\dots,d_r))>0$.
\end{proof}

\section{Calculation of classical Hirzebruch genera of complete intersections}\label{sec-Hir-genus}

In this section, for the power series $Q(x)=x/R(x)$, we discuss the associated Hirzebruch genus $\varphi^{}_Q$ of complete intersections $X_n(\underline{d})$ with multi-degree $\underline{d}=(d_1,\dots,d_r)$. Then for certain given power series $Q(x)$, we calculate the associated classical Hirzebruch genera of complete intersections as examples.

\begin{theorem}\label{HGCI}
The Hirzebruch genus $\varphi^{}_Q$ on the complete intersection $X_n(\underline{d})$ satisfies that:
\begin{equation*}
\varphi^{}_Q(X_{n}(\underline{d}))=\frac{1}{2\pi\sqrt{-1}}\oint_{\varGamma(0)}\frac{\prod_{i=1}^r{R(d_iz)}}{{(R(z))}^{n+r+1}}d z.
\end{equation*}
\end{theorem}
\begin{proof} 
Let $x\in H^2(\cp^{n+r};\zz)$ be the first Chern class of the Hopf line bundle $H$ over $\cp^{n+r}$.
By \Cref{VRKCI} and \Cref{VK},
\begin{align*}
\varphi^{}_Q(X_{n}(\underline{d}))&= \varphi^{}_Q(d_1x,\dots,d_rx)_{\mathbb{C}P^{n+r}} \\
&= \int_{\mathbb{C}P^{n+r}}{\prod_{i=1}^rR(d_ix)\cdot \varphi^{}_Q(T(\cp^{n+r}))}.  
\end{align*}
The total Chern class of $\cp^{n+r}$ is $c(\cp^{n+r})=(1+x)^{n+r+1}$, so
\begin{equation*}
\varphi^{}_Q(T(\cp^{n+r}))=(Q(x))^{n+r+1}.
\end{equation*}
Thus
\begin{align*}
\varphi^{}_Q(X_{n}(\underline{d}))=& \int_{\mathbb{C}P^{n+r}}{\prod_{i=1}^rR(d_ix)\cdot (Q(x))^{n+r+1}} \\ 
=& \int_{\mathbb{C}P^{n+r}}x^{n+r+1}\frac{\prod_{i=1}^r{R(d_ix)}}{{(R(x))}^{n+r+1}}. 
\end{align*}
By residue theorem, 
\[
\varphi^{}_Q(X_{n}(\underline{d}))=\frac{1}{2\pi\sqrt{-1}}\oint_{\varGamma(0)}\frac{\prod_{i=1}^r{R(d_iz)}}{{(R(z))}^{n+r+1}}d z.  \qedhere
\]
\end{proof}

\begin{example}\label{Tygenus}
For $T_y$-genus (or generalized Todd genus) in \cite[page 93]{Hi1978}, the associated power series is   $Q(y;x)=\dfrac{x}{R(y;x)}$, where
\[
R(y;x)=\frac{\exp(x(y+1))-1}{\exp(x(y+1))+y}.
\]
Hence by \Cref{HGCI} we have
\begin{align}\label{TyCI}
&~~~~ T_y(X_{n}(\underline{d})) =\frac{1}{2\pi\sqrt{-1}}\oint_{\varGamma(0)}\frac{\prod_{j=1}^r{R(y;d_jz)}}{{(R(y;z))}^{n+r+1}}d z \nonumber\\
&= \frac{1}{2\pi\sqrt{-1}}\oint_{\varGamma(0)}{\prod_{j=1}^r{\frac{\exp(d_jz(y+1))-1}{\exp(d_jz(y+1))+y}}\cdot{\left(\frac{\exp(z(y+1))+y}{\exp(z(y+1))-1}\right)}^{n+r+1}}d z.
\end{align}
\end{example}
\begin{example}
In \Cref{Tygenus}, when $y=0$,  $R(y;x)$ is $R(0;x)=1-\exp(-x)$. The $T_0$-genus is the Todd genus, which has the following form:
\begin{align}\label{T1} 
{\rm Td}(X_{n}(\underline{d})) & =\frac{1}{2\pi\sqrt{-1}}\oint_{\varGamma(0)}\frac{\prod_{j=1}^r(1-\exp{(-d_jz)})}{{(1-\exp{(-z)})}^{n+r+1}}d z \nonumber\\
& =\frac{1}{2\pi\sqrt{-1}}\oint_{\varGamma(0)}\frac{\prod_{j=1}^r(1-\exp{(-d_jz)})}{{(1-\exp{(-z)})}^{n+r+1}}\cdot\exp(z)d(1-\exp{(-z)}) \nonumber\\
& =\frac{1}{2\pi\sqrt{-1}}\oint_{\varGamma(0)}\frac{\prod_{j=1}^r\left(1-(1-\omega)^{d_j}\right)}{\omega^{n+r+1}}\cdot(1-\omega)^{-1}d \omega \nonumber\\
& =\frac{1}{2\pi\sqrt{-1}}\oint_{\varGamma(0)}\frac{1}{\omega^{n+r+1}}\sum_{j=0}^r{{(-1)^{j}}\sum_{1\leqslant{k_1}<\dots<{k_j}\leqslant r}(1-\omega)^{-1+d_{k_1}+\cdots+d_{k_j}}}d \omega \nonumber\\
& =\sum_{j=0}^r{{(-1)^{j}}\sum_{1\leqslant{k_1}<\dots<{k_j}\leqslant r}(-1)^{n+r}\dbinom{-1 +d_{k_1}+\dots+d_{k_j}}{n+r}} \nonumber\\
& =\sum_{j=0}^r{{(-1)^{n+r+j}}\sum_{1\leqslant{k_1}<\dots<{k_j}\leqslant r}\dbinom{-1 +d_{k_1}+\dots+d_{k_j}}{n+r}}. 
\end{align}

Executing another process, we have the following result:
\begin{align}\label{T2}
{\rm Td}(X_{n}(\underline{d})) & =\frac{1}{2\pi\sqrt{-1}}\oint_{\varGamma(0)}\frac{\prod_{j=1}^r(1-\exp{(-d_jz)})}{{(1-\exp{(-z)})}^{n+r+1}}d z \nonumber\\
& =\frac{1}{2\pi\sqrt{-1}}\oint_{\varGamma(0)}\exp{(c_1z)}\frac{\prod_{j=1}^r(\exp{(d_jz)}-1)}{{(\exp{(z)}-1)}^{n+r+1}}d z \nonumber\\
& =\frac{1}{2\pi\sqrt{-1}}\oint_{\varGamma(0)}\exp{(c_1z)}\frac{\prod_{j=1}^r(\exp{(d_jz)}-1)}{{(\exp{(z)}-1)}^{n+r+1}}\cdot\exp(-z)d(\exp{(z)}-1) \nonumber\\\
& =\frac{1}{2\pi\sqrt{-1}}\oint_{\varGamma(0)}\frac{\prod_{j=1}^r\left((1+\omega)^{d_j}-1\right)}{\omega^{n+r+1}}\cdot(1+\omega)^{c_1-1}d \omega \nonumber\\\
& =\sum_{j=0}^r{{(-1)^{r-j}}\sum_{1\leqslant{k_1}<\dots<{k_j}\leqslant r}\dbinom{c_1-1 +d_{k_1}+\dots+d_{k_j}}{n+r}}, 
\end{align}
where $c_1=n+r+1-\sum\limits_{i=1}^rd_i$. 

In fact, \eqref{T1} and \eqref{T2} are equivalent.
Since $\dbinom{a}{k}$ is a generalized binomial coefficients coventioned as in \Cref{Ahatgenus}, it is easy to know $(-1)^k\dbinom{a}{k}=\dbinom{-a+k-1}{k}$,  then we have
\begin{align*}
& ~~~~~(-1)^{n+r}\cdot\dbinom{-1+d_{k_1}+\cdots+d_{k_j}}{n+r} \\
&= \dbinom{n+r-1+1-(d_{k_1}+\cdots+d_{k_j})}{n+r}\\
&= \dbinom{(n+r+1-\sum_{i=1}^rd_i)-1+\sum_{i=1}^rd_i-(d_{k_1}+\cdots+d_{k_j})}{n+r}\\
&= \dbinom{c_1-1+(\sum_{i=1}^rd_i-(d_{k_1}+\cdots+d_{k_j}))}{n+r}.
\end{align*}
So \eqref{T1} implies that 
\begin{align*}
{\rm Td}(X_{n}(\underline{d}))& =\sum_{j=0}^r{{(-1)^{j}}\sum_{1\leqslant{k_1}<\dots<{k_j}\leqslant r}(-1)^{n+r}\dbinom{-1 +d_{k_1}+\dots+d_{k_j}}{n+r}} \\
& =\sum_{j=0}^r{{(-1)^{j}}\sum_{1\leqslant{k_1}<\dots<{k_j}\leqslant r}\dbinom{c_1-1+(\sum_{i=1}^rd_i-(d_{k_1}+\cdots+d_{k_j}))}{n+r}} \\
& =\sum_{j=0}^r{{(-1)^{r-j}}\sum_{1\leqslant{k_1}<\dots<{k_j}\leqslant r}\dbinom{c_1-1 +d_{k_1}+\dots+d_{k_j}}{n+r}}.
\end{align*}
Thus, the two forms \eqref{T1} and \eqref{T2} about Todd genus of $X_n(\underline{d})$ are equivalent.
\end{example}
\begin{example}
In \Cref{Tygenus}, when $y=-1$, $R(y;x)$ is
\begin{align*}
R(-1;x)& =\lim_{y\to-1}\frac{\exp(x(y+1))-1}{\exp(x(y+1))+y}\\ &=\lim_{y\to-1}\frac{x\exp(x(y+1))}{x\exp(x(y+1))+1}\\
&=\frac{x}{1+x},
\end{align*}
and \eqref{TyCI} is the Euler characteristic of complete intersection $X_{n}(\underline{d})$:
\[
\chi(X_{n}(\underline{d}))=\frac{1}{2\pi\sqrt{-1}}\oint_{\varGamma(0)}\frac{\prod_{j=1}^r\frac{d_jz}{1+d_jz}}{\left(\frac{z}{1+z}\right)^{n+r+1}}d z .
\]
Let $\omega=\dfrac{z}{1+z}$, then $z=\dfrac{\omega}{1-\omega}$. It implies that 
\begin{align*}
\dfrac{d_jz}{1+d_jz} &= \dfrac{d_j\omega}{1+(d_j-1)\omega},\\
d \omega &= (1-\omega)^2d z ~(\text{the differential of $\omega$}).
\end{align*}
 Hence,
\begin{align*}
\chi(X_{n}(\underline{d}))& =\frac{1}{2\pi\sqrt{-1}}\oint_{\varGamma(0)}\prod_{j=1}^r\frac{d_j\omega}{1+(d_j-1)\omega}\cdot\frac{1}{{\omega}^{n+r+1}}\cdot\frac{1}{{(1-\omega)}^{2}}d \omega \\
& =\prod_{j=1}^rd_j\cdot\frac{1}{2\pi\sqrt{-1}}\oint_{\varGamma(0)}\frac{1}{\prod_{j=1}^r{(1+(d_j-1)\omega)}\cdot{(1-\omega)}^{2}}\cdot\frac{1}{{\omega}^{n+1}}d \omega \\
& =\prod_{j=1}^rd_j\cdot h_n(1,1,1-d_1,\dots,1-d_r), 
\end{align*}
where $h_{n}(a_{1},\ldots, a_{k})=\sum\limits_{1 \leqslant i_{1}\leqslant\cdots\leqslant i_{n} \leqslant k} a_{i_{1}} a_{i_{2}} \cdots a_{i_{n}}$ is the coefficient of $z^{n}$ in $\prod\limits_{j=1}^{k} \dfrac{1}{1-a_{j} z}$, and $h_{n}(a_{1},\ldots, a_{k})$ is called a complete homogeneous symmetric polynomial of degree $n$ on indeterminates $a_{1}, \ldots, a_{k}$.
\end{example}
\begin{example}
In \Cref{Tygenus}, when $y=1$, $R(y;x)$ is  $R(1;x)=\tanh x$ and \eqref{TyCI} is the corresponding signature  of complete intersection $X_{n}(\underline{d})$.
\begin{align*}
&~~~~~ \tau(X_{n}(\underline{d})) \\
&= \frac{1}{2\pi\sqrt{-1}}\oint_{\varGamma(0)}\frac{\prod_{j=1}^r\tanh(d_jz)}{{(\tanh z)}^{n+r+1}}d z \\
&= \frac{1}{2\pi\sqrt{-1}}\oint_{\varGamma(0)}{\prod_{j=1}^r\frac{\exp{(2d_jz)}-1}{\exp{(2d_jz)}+1}}\cdot{{\left(\frac{\exp{(2z)}+1}{\exp{(2z)}-1}\right)}^{n+r+1}}d z \\
&= \frac{1}{2\pi\sqrt{-1}}\oint_{\varGamma(0)}{\prod_{j=1}^r\frac{\exp{(2d_jz)}-1}{\exp{(2d_jz)}+1}}\cdot{{\left(\frac{\exp{(2z)}+1}{\exp{(2z)}-1}\right)}^{n+r+1}}\cdot \frac{\exp(-2z)}{2}d (\exp(2z)-1) \\
&= \frac{1}{2\pi\sqrt{-1}}\oint_{\varGamma(0)}{\prod_{j=1}^r\frac{(1+\omega)^{d_j}-1}{(1+\omega)^{d_j}+1}}\cdot{{\left(\frac{2+\omega}{\omega}\right)}^{n+r+1}}\cdot\frac{1}{2(1+\omega)}d \omega.
\end{align*}
\end{example}
\begin{example}
For Witten genus \cite[page 82]{HiBeJu1992}, the power series is that  
\[
Q(x)=\dfrac{x}{R(L;x)}=\dfrac{x}{\sigma_L(x)}, 
\]
where $R(L;x)=\sigma_L(x)$ is the Weierstrass $\sigma$-function for Lattice $L$.
Hence, by \Cref{HGCI}, the Witten genus of $X_n(\underline{d})$ is
\begin{equation*}
W(X_{n}(\underline{d})) =\frac{1}{2\pi\sqrt{-1}}\oint_{\varGamma(0)}\frac{\prod_{j=1}^r{\sigma_L(d_jz)}}{{(\sigma_L(z))}^{n+r+1}}d z.
\end{equation*}
\end{example}
\begin{example}
For $\hat{A}$-genus, the power series corresponding to $\hat{A}$-genus is
\begin{equation*}
Q(x)=\frac{x}{R(x)}=\frac{\frac{1}{2}x}{\sinh\left(\frac{1}{2}x\right)}, \text{ where }
R(x)=2\sinh\left(\frac{1}{2}x\right).
\end{equation*}
Then by \Cref{HGCI},
\begin{align*}
&~~~~~ \hat{A}(X_{2n}(d_1,\dots,d_r)) \\
&= \frac{1}{2\pi\sqrt{-1}}\oint_{\varGamma(0)}\frac{\prod_{j=1}^r\left(2\sinh\left(\frac{1}{2}d_jz\right)\right)}{{\left(2\sinh\left(\frac{1}{2}z\right)\right)}^{2n+r+1}}d z \\
&= \frac{1}{2\pi\sqrt{-1}}\oint_{\varGamma(0)}\frac{\prod_{j=1}^r\left(\exp{\left(\frac{1}{2}d_jz\right)}-\exp{\left(-\frac{1}{2}d_jz\right)}\right)}{{\left(\exp{\left(\frac{1}{2}z\right)}-\exp{\left(-\frac{1}{2}z\right)}\right)}^{2n+r+1}}d z \\
&= \frac{1}{2\pi\sqrt{-1}}\oint_{\varGamma(0)}\exp{\left(\frac{1}{2}\big(2n+r+1-\sum_{i=1}^rd_i\big)z\right)}\cdot\frac{\prod_{j=1}^r(\exp{(d_jz)}-1)}{{(\exp{(z)}-1)}^{2n+r+1}}d z \\
&= \frac{1}{2\pi\sqrt{-1}}\oint_{\varGamma(0)}\exp{\left(\frac{1}{2}c_1z\right)}\cdot\frac{\prod_{j=1}^r(\exp{(d_jz)}-1)}{{(\exp{(z)}-1)}^{2n+r+1}}\cdot\exp{(-z)}d(\exp{(z)}-1) \\
&= \frac{1}{2\pi\sqrt{-1}}\oint_{\varGamma(0)}(1+\omega)^{\frac{1}{2}c_1-1}\cdot\frac{\prod_{j=1}^r\left((1+\omega)^{d_j}-1\right)}{\omega^{2n+r+1}}d \omega \\
&= \sum_{j=0}^r{{(-1)^{r-j}}\sum_{1\leqslant{k_1}<\dots<{k_j}\leqslant r}\dbinom{\frac{1}{2}c_1-1 +d_{k_1}+\dots+d_{k_j}}{2n+r}}, 
\end{align*}
where $c_1=2n+r+1-\sum\limits_{i=1}^rd_i$.
\end{example}
\begin{example}In \cite{Kricever1976}, Kri$\check{c}$ever proved that the values of the Hirzebruch genera $A_k$, $k= 2, 3, \dots,$ are obstructions to the existence of nontrivial
$S^1$-actions on a unitary manifold whose first Chern class is divisible by $k$.
For Hirzebruch genera $A_k$ in \cite[\S 2]{Kricever1976}, the associated power series is
\begin{equation*}\label{AkRx}
Q(x)=\frac{x}{R(x)}=\frac{kx\exp(x)}{\exp(kx)-1},  k= 2, 3, \dots,
\end{equation*}
where $R(x)=\dfrac{\exp(kx)-1}{k\exp(x)}$.

Then by \Cref{HGCI},
\begingroup
\allowdisplaybreaks
\begin{align*}
&~~~~~ A_k(X_{n}(\underline{d}))=\frac{1}{2\pi\sqrt{-1}}\oint_{\varGamma(0)}\frac{\prod_{i=1}^r\frac{\exp{(kd_iz)}-1}{k\exp(d_iz)}}{{\left(\frac{\exp{(kz)}-1}{k\exp(z)}\right)}^{n+r+1}}d z \\
&= \frac{1}{k}\cdot \frac{1}{2\pi\sqrt{-1}}\oint_{\varGamma(0)}{{\prod_{i=1}^r\frac{\exp{(d_iz)}-1}{k\exp{\left(\frac{1}{k}d_iz\right)}}}\cdot{\left(\frac{k\exp{\left(\frac{1}{k}z\right)}}{\exp{(z)}-1}\right)^{n+r+1}}}d z \\
&= \frac{1}{k}\cdot \frac{1}{2\pi\sqrt{-1}}\oint_{\varGamma(0)}{{\prod_{i=1}^r\frac{\exp{(d_iz)}-1}{k\exp{\left(\frac{1}{k}d_iz\right)}}}\cdot{\left(\frac{k\exp{\left(\frac{1}{k}z\right)}}{\exp{(z)}-1}\right)^{n+r+1}}}\exp(-z)d(\exp{(z)}-1) \\
&= \frac{1}{k}\cdot \frac{1}{2\pi\sqrt{-1}}\oint_{\varGamma(0)}{{\prod_{i=1}^r\frac{(1+\omega)^{d_i}-1}{k(1+\omega)^{\frac{1}{k}d_i}}}\cdot{\left(\frac{k(1+\omega)^{\frac{1}{k}}}{\omega}\right)^{n+r+1}}\cdot{(1+\omega)}^{-1}}d \omega \\
&= k^n\cdot \frac{1}{2\pi\sqrt{-1}}\oint_{\varGamma(0)}{\frac{\prod_{i=1}^r((1+\omega)^{d_i}-1)}{\omega^{n+r+1}}\cdot{(1+\omega)}^{\frac{1}{k}\left(n+r+1-\sum_id_i\right)-1}}d \omega \\
&= k^n\cdot \frac{1}{2\pi\sqrt{-1}}\oint_{\varGamma(0)}{\frac{\prod_{i=1}^r((1+\omega)^{d_i}-1)}{\omega^{n+r+1}}\cdot{(1+\omega)}^{\frac{1}{k}c_1-1}}d \omega \\
&= k^n\sum_{j=0}^r{(-1)^{r-j}\sum_{1\leqslant k_1<\dots<k_j\leqslant r}\dbinom{\frac{1}{k}c_1-1+d_{k_1}+\dots+d_{k_j}}{n+r}}, 
\end{align*}
\endgroup
where $c_1=n+r+1-\sum\limits_{i=1}^rd_i$. Note that $A_1(X_{n}(\underline{d}))$ is the Todd genus ${\rm Td}(X_{n}(\underline{d}))$,  and  $A_2(X_{n}(\underline{d}))$ coincides with the $\hat{A}$-genus $\hat{A}(X_{n}(\underline{d}))$  up to a factor $2^n$.
\end{example}

\end{document}